\documentclass[11pt]{amsart}
\usepackage[margin={0.7in}, bmargin={0.7in}, tmargin={0.7in}  ]{geometry}
\usepackage{amsfonts}
\usepackage{amsmath}
\usepackage{epsfig}
\usepackage{color}
\usepackage{graphicx}
\usepackage{psfrag}
\usepackage{eufrak}
\usepackage{subfigure}
\usepackage{geometry}
\usepackage{mathtools}
\usepackage{mhequ}
\usepackage{comment}

\numberwithin{equation}{section}

\newcommand{\be}{\begin{equation}}
\newcommand{\ee}{\end{equation}}

\DeclareMathOperator{\Ric}{Ric}
\DeclareMathOperator{\divergence}{div}
\DeclareMathOperator{\trace}{tr}
\DeclareMathOperator{\vol}{Vol}

\DeclareMathOperator{\Ker}{Ker}
\DeclareMathOperator{\Length}{Lenght}

\newcommand{\calA}{{\mathcal A}}

\newcommand{\second}{I\!I}

\numberwithin{equation}{section}

\newtheorem{lemma}{Lemma}[section]

\newtheorem{theorem}[lemma]{Theorem}
\newtheorem{proposition}[lemma]{Proposition}
\newtheorem{corollary}[lemma]{Corollary}
\theoremstyle{remark}
\newtheorem{remark}[lemma]{Remark}
\newtheorem*{conjecture}{Conjecture}

\begin{document}

\title[Eigenvalue bounds for Paneitz with boundary]{Eigenvalue bounds for the Paneitz operator and its associated third-order boundary operator on locally conformally flat manifolds}

\author[M.d.M. Gonz\'alez]{Mar\'ia del Mar Gonz\'alez}

\address{Mar\'ia del Mar Gonz\'alez
\hfill\break\indent
Universidad Aut\'onoma de Madrid
\hfill\break\indent
Departamento de Matem\'aticas, and ICMAT, 28049 Madrid, Spain}
\email{mariamar.gonzalezn@uam.es}

\author[M. S\'aez]{Mariel S\'aez}

\address{Mariel S\'aez
\hfill\break\indent
P. Universidad Cat\'olica de Chile
\hfill\break\indent
Departamento de Matem\'aticas, Santiago, Chile}
\email{mariel@mat.uc.cl}

\maketitle

\begin{abstract}
In this paper we study bounds for the first eigenvalue of the Paneitz operator $P$ and its associated third-order boundary operator $B^3$ (see \eqref{definition P} and \eqref{Paneitz-boundary} for a precise definitions) on four-manifolds. We restrict to orientable, simply connected, locally confomally flat manifolds that have at most two umbilic boundary components.  The proof is based on showing that under the hypotheses of the main theorems, the considered manifolds are confomally equivalent to canonical models. This equivalence is proved by showing the injectivity of suitable developing maps. Then the bounds on the eigenvalues are obtained through explicit computations on the canonical models and its connections with
the classes of manifolds that we are considering. In particular, we give an explicit bound for a 4-dimensional annulus with a radially symmetric metric. The fact that $P$ and $B^3$ are conformal in four dimensions is key in the proof.

\end{abstract}
\section{Introduction}

Let $(M^4,g)$ be a 4-dimensional Riemannian manifold and denote by $\Ric$   and $W$  the Ricci and Weyl tensors of $g$, respectively. Define $J$ to be  the trace of the Schouten tensor $A = \frac{1}{2}(\Ric-Jg)$ (actually $J$ is a multiple of the scalar curvature $R$, this is, $J=\frac{1}{6}R$) and $dv_g$ the volume element for the metric $g$.\\

 The Paneitz operator  $P_g$ on $(M,g)$, first introduced in 1983 \cite{Paneitz83}, is defined by
\begin{equation}\label{definition P}
P_g=(-\Delta_g)^2+\divergence_g\big\{4A_g-2Jg\big\}d,
\end{equation}
$P$ is a conformally covariant operator and, in particular, it satisfies that under a change of metric $g_f=e^{2f}g$,
\begin{equation}\label{conformal-Paneitz}
P_{g_f}=e^{-4f}P_{g} \quad\text{on }M.
\end{equation}
This operator describes the transformation law for Branson's $Q$-curvature \cite{Branson85}, which is defined by
\begin{equation*}
Q_g= \tfrac{1}{6}(-\Delta R_g-3|\Ric_g\!|^2 + R_g^2).
\end{equation*}
Indeed,
\begin{equation*}
P_{g} f+Q_{g_f} e^{4f}=Q_g,\quad\text{for}\quad g_f=e^{2f}g.
\end{equation*}
There is an extensive bibliography on the  $Q$-curvature equation in dimension four. Without being exhaustive, we mention
 \cite{Chang-Yang:extremal,Wei-Xu,Djadli-Malchiodi,Li-Li-Liu,Gursky-Malchiodi}.

Now, if $M$ is a compact 4-dimensional manifold without boundary, the Chern-Gauss-Bonnet formula \cite{Branson:sharp-inequalities} reads
\begin{equation}\label{Gauss-Bonnet}
\int_M Q_g \,dv_g+\frac{1}{4}\int_M |W_g|_g^2\,dv_g=8\pi^2\chi(M).
\end{equation}

Note then that we may regard $P$ as a generalization for 4-manifolds of the Laplace operator $\Delta$ in two dimensions (that is also conformally covariant in that setting) and the curvature $Q$ as a four-dimensional analog of the Gaussian curvature in the two-dimensional setting (which plays the same role as $Q$ in the classical Gauss-Bonnet formula).\\

The study of eigenvalues of differential operators has an extensive history. In the particular case of the Laplacian in two dimensions it is possible to obtain bounds that only depend on the topology of the manifold (see  \cite{Yang-Yau}); more precisely, consider  $N^2$ to be a two-dimensional compact orientable Riemanniann manifold with no boundary and take $\varsigma_1$ to be the first non-zero eigenvalue of the Laplace-Beltrami on $N$. Let
\begin{equation}\label{Lambda}
\Theta(N^2):=\sup\{\varsigma_1  \vol(N^2)\},
\end{equation}
where the $\sup$ is taken among all Riemannian metrics on $N^2$. It is well known that $\Theta(N^2)<\infty$ and
\begin{equation}\label{bound:surfaces}
\Theta(N^2)\leq 8\pi(\gamma+1),
\end{equation}
where $\gamma$ is the genus of the surface. See also \cite{Nadirashvili}  for a discussion of attainability. In fact, in that paper the discussion is  extended to the setting with boundary and a Neumann condition at that boundary.\\

The first goal in our paper is to generalize the bound \eqref{bound:surfaces} for the Paneitz operator $P$ on closed 4-manifolds. However, although it is well known that the spectrum of $P_g$  consists of a sequence of eigenvalues converging to $+\infty$, the principal eigenvalue $\lambda_0^g$ maybe negative. In consequence, one first imposes restrictions that ensure the positivity of the operator.  With  this objective, we recall   two important conformal invariant quantities in four dimensions: Firstly, the \emph{total  $Q$-curvature},
\begin{equation}\label{kappa}
\kappa_g:=\int_M Q_g \,dv_g,
\end{equation}
and, secondly,  the well known \emph{Yamabe invariant}
\begin{equation}\label{Yamabe-invariant-closed}
\mathcal Y[g]=\inf_{g_f=e^{2f}g} \frac{\int_M R_{g_{f}}\,dv_{g_f}}{\left(\int_M dv_{g_f}\right)^{1/2}}.
\end{equation}

A  key theorem by Gursky \cite{Gursky:principal-eigenvalue} yields that, if both the Yamabe invariant $Y[M]$  and the total $Q-$curvature $\kappa_g$ are nonnegative, then $\lambda_0^g=0$ and the kernel of $P_g$ contains only the constant functions. Thus, the next eigenvalue $\lambda_1^g$ is positive. A less restrictive condition was given in \cite{Gursky-Viaklovsky:fully-nonlinear-equation}: indeed, if $M$ is a closed 4-manifold with positive scalar curvature and
 \begin{equation}\label{Gursky-condition}
 \int_M Q_g\,dv_g +\frac{1}{3}(\mathcal Y[g])^2>0,
 \end{equation}
then the same conclusion holds.  It is interesting to observe that \eqref{Gursky-condition}  is a conformally invariant quantity. Now, the first eigenvalue of  $P_g$ (that we denote as $\lambda_1^g$) can be computed through the Rayleigh quotient
\begin{equation}\label{Rayleigh}
\lambda_1^g=\inf_{\int_M u\,dv_g=0,\\ u\ne 0} \,\,\frac{\mathcal E^M_g[u]}{\int_M u^2\,dv_g},
\end{equation}
where
\begin{equation}\label{Eg}
\mathcal E^M_g[u]=\int_{M} (\Delta_g u)^2\,dv_g+\int_M \Big(2Jg_{ab}-4A^g_{ab}\Big)\nabla^a u\nabla^b u\,dv_g.
\end{equation}
This is a conformal invariant quantity in 4-dimensions; indeed, if we have two metrics related by $g_f=e^{2f}g$, then
\begin{equation*}
\mathcal E_{g_f}[u]=\mathcal E_g[u].
\end{equation*}

Our first theorem is a generalization of the bound \eqref{bound:surfaces} for the first eigenvalue of the Paneitz operator:

\begin{theorem}\label{thm:closed} Let $(M,g)$ be a compact, orientable, closed,  locally conformally flat ({\it l.c.f}) Riemannian 4-manifold, and define $\lambda_1^g$ to be the first (non-zero) eigenvalue of the Paneitz operator $P_g$. Then:
\begin{itemize}
\item[\emph{i.}] If $M$ is simply connected, then $M$ is conformally equivalent to $\mathbb S^4 $. In this setting we have $\lambda_1>0$ with $\Ker(P_g)=\{\text{constants}\}$ and
\begin{equation*}
\lambda_1^g \, {\vol(M)}\leq 64\pi^2.
\end{equation*}
Equality holds if and only if $M$ is diffeomorphic to $\mathbb S^4$.

\item[\emph{ii.}]  If $\mathcal Y[g]>0$ and $\kappa_g>0$, and $M$ is orientable, then $M$ is conformally equivalent to $\mathbb S^4$ and the same conclusions  hold.

\item[\emph{iii.}] If $\mathcal Y[g]>0$ and $\kappa_g=0$, then $M$ is conformally equivalent to a quotient  $\mathbb R\times \mathbb S^3$. 
    Then $\lambda_1^g>0$ with $\Ker(P_g)=\{\text{constants}\}$.

 Assume that the fundamental domain is exactly $[0,\varrho) \times\mathbb{S}^3$ for some $\varrho>0$ and let $\Psi: M\to [0,\varrho) \times\mathbb{S}^3$ be the conformal embedding described above. Set $\Psi_{\mathbb{S}^3}$ to be its projection onto the $\mathbb S^3$-coordinates. If, in addition, we impose the
geometric condition that for all $q\in M$, there exist  $\delta_0\ge 0$ and $\varepsilon\in (0,1) $ such that, for every $\delta<\delta_0$ it holds \begin{equation}\label{concentrating}\frac{\vol_M( \mathcal B_\delta(\Psi_{\mathbb{S}^3}(q))}{\vol_M( \mathcal B^c_\delta(\Psi_{\mathbb{S}^3}(q))}<\varepsilon,\end{equation} then
\begin{equation*}
\lambda_1^g{\vol(M)}\leq  C(\varepsilon, \delta_0) \varrho.
\end{equation*}
Here $C(\varepsilon, \delta_0)$ is a constant that only depends on $\varepsilon, \delta_0$, while
$\mathcal B_\delta(\cdot)$ is the geodesic ball on $\mathbb{S}^3$ with the standard metric centered at $\Psi_{\mathbb{S}^3}(q)$ and $\mathcal B^c_\delta(\cdot)$ its complement in $\mathbb{S}^3$. We also denoted $\vol_M(A)=\int_{M\cap \Psi^{-1}(A)} dv_g$.

\end{itemize}
\end{theorem}
Two remarks regarding statement \emph{iii.} are in order:
\begin{itemize}
    \item    The geometric condition \eqref{concentrating} can be understood as a quantitative measure that avoids concentration around  lines.
\item The exact value of the constant $C(\alpha, \delta_0)$ can be calculated precisely, but it is cumbersome and does not provide additional information. It is worth noting, though, that it blows up if
certain parameter $\delta$ approaches 0, but this is ruled out by condition \eqref{concentrating}.

\end{itemize}

The jump from dimension 2 to dimension 4 is completely non-trivial, since in the two-dimensional case one can use conformal invariance to map any manifold to (a cover of) the sphere. In dimension 4 the difficulty is to find such conformal immersion of $M$ into a model manifold. Thus we restrict our study to locally conformally flat (l.c.f.) manifolds, where the {\it developing map} plays the role of this immersion. In fact, we will show that in the setting of Theorem \ref{thm:closed} the developing map is injective.

The idea of the proof of Theorem \ref{thm:closed} follows  Hersch' original idea in \cite{Hersch}. Indeed,  one uses a calibration type argument in order to show that coordinate functions are good test functions for the Rayleigh quotient \ref{Rayleigh}. In the case that $M$ is conformally equivalent to the sphere, this calculation has also been performed in \cite{PerezAyala}. The case of the cylinder $\mathbb R\times\mathbb S^3$ is much trickier but a variation of Hersch calibration argument can  still be performed.

Related to this result is is the work of \cite{PerezAyala}, where the author shows that for the extremal metric for the quantity $\lambda_1^g\vol(M)$ one may obtain an orthonormal basis of eigenfunctions. In addition, these eigenfunctions are the coordinates for a Paneitz map, which is a 4-th order generalization of a harmonic map.  \\

Finally, note that different bounds  for $\lambda_1^g$ have been introduced in the literature. For instance, if $M$ can be conformally immersed into a unit sphere $\mathbb S^K$, then \cite{Xu-Yang:conformal-energy} gives a bound
in terms of  an $K$-conformal energy inspired in the conformal volume of
Li-Yau \cite{Li-Yau}. In addition, \cite{Chen-Li:first-eigenvalue,Cheng:eigenvalues} showed some geometric bounds provided that  $M$ is a compact submanifold of $\mathbb R^K$.  A comparison theorem for this first eigenvalue was given in \cite{Wang-Zhou:comparison} for dimension $K\geq 5$ in some settings.\\

\subsection{Manifolds with boundary}\label{section:introduction-boundary}

Now we turn our attention to the boundary case.  If $N^2$ is a compact surface with boundary, one may ask the same questions for the Steklov eigenvalues, which are the eigenvalues $\vartheta $ of the following the boundary value problem
\begin{equation*}
\left\{\begin{split}
&\Delta_g u=0 \text{ in }N,\\
&-\partial_\eta u=\vartheta u \text{ on }\partial N.
\end{split}\right.
\end{equation*}
A good reference for this problem is \cite{GP}. Given $N$, is well known that there exists an increasing sequence of eigenvalues $0=\vartheta_0<\vartheta_1\leq \vartheta_2\leq \ldots$
and, moreover,
\begin{equation*}
\vartheta_1=\inf_{\int_{\partial N}u=0,u\neq 0} \frac{\int_N |\nabla u|^2\,dv}{\int_{\partial N} u^2\,d\sigma}.
\end{equation*}
The extremal problem for the Steklov eigenvalue analogous to \eqref{Lambda} has been studied in a series of papers by Fraser-Schoen \cite{Fraser-Schoen:annulus,Fraser-Schoen:survey,Fraser-Schoen:free-boundary-minimal-surfaces}. If $N^2$ is a surface of genus $\gamma$ and $k$ boundary components, they show the bound
\begin{equation}\label{Steklov}
\vartheta_1(N) \Length(\partial N)\leq 2\pi( \gamma+k).
\end{equation}
For $\gamma = 0$ and $k = 1$ this result was obtained by Weinstock \cite{Weinstock} and it is sharp, while if the boundary has two boundary components (i.e., an annulus), it is not attained.
 In addition, Weinstock  showed that the bound is attained at a flat disk and the eigenfunctions can be identified with its coordinates.  In the general case,
  Fraser-Schoen~\cite{Fraser-Schoen:survey,Fraser-Schoen:free-boundary-minimal-surfaces}  identified the eigenfunctions associated to maximal eigenvalues (with a given topology and number of boundary components) with coordinate functions of free boundary minimal surfaces in the unit ball $\mathbb B^K$.
 In the particular case that $N$ is homeomorphic to the annulus, in \cite{Fraser-Schoen:annulus} and \cite{Fraser-Schoen:free-boundary-minimal-surfaces}  it is shown that the quantity $\vartheta_1(N) \Length(\partial N)$  is maximized  by the coordinate functions of a critical  catenoid (in $\mathbb R^3$) which meets the boundary sphere orthogonally.
This problem has also been studied in the higher dimensional setting \cite{Fraser-Schoen:higher-dimension}, where conformal invariance is lost and the maximizer does not exist in the class of smooth metrics.\\

In this paper we are interested in the analog question for a conformal third-order boundary operator associated to the Paneitz operator, and which yields the natural 4-dimensional generalization of the Steklov problem from the conformal geometry point of view. In addition, it contains strong topological information thanks to the Chern-Gauss-Bonnet formula given in formula \eqref{CGB-full} (and the discussion above it).
It  was introduced in \cite{Chang-Qing:zeta1,Chang-Qing:zeta2} (see also the surveys: \cite{Chang:survey,Chang:survey-CRM}, for instance), and fully generalized in \cite{Case:boundary-operators}. We follow the presentation in the latter.\\

Set $(M^4,g)$ be a 4-dimensional compact Riemannian manifold with boundary $\Sigma=\partial M$. We keep the notation above for the interior quantities, while tilde will mean the corresponding quantities for the boundary metric. Denote by $h$ the restriction of the metric $g$ to $T\Sigma$ and by $d\sigma_h$ the volume form for  $h$ on $\Sigma$. Let $\eta$ be the outward-pointing normal, $\second=\nabla\eta|_{T\Sigma}$ the second fundamental form,  $H = \trace_h \nabla\eta$  the mean curvature of $\Sigma$, and $\second_0=\second-\frac{H}{3}h$ the trace free part of the second fundamental form.\\

We set, on the boundary $\Sigma$:
\begin{equation*}
\begin{split}
&B_g^1  u= \eta u,\\
&B_g^2 u= -\tilde\Delta u+D^2 u (\eta,\eta)+\frac{1}{3}H\eta u,
\end{split}
\end{equation*}
and the third order operator
\begin{equation}\label{Paneitz-boundary}
B_g^{3} u =-\eta\Delta u-2\tilde \Delta \eta u + 2\langle \second_0,\tilde D^2 u\rangle -\frac{2}{3}H\tilde \Delta u+\frac{2}{3}\langle\tilde\nabla H,\tilde \nabla u\rangle+\Big(-\frac{1}{3}H^2-2A(\eta,\eta)+2\tilde J+\frac{1}{2}|\second_0|^2\Big)\eta u,
\end{equation}
 These operators also satisfy a conformally covariance property coupled with \eqref{conformal-Paneitz}, this is
\begin{equation}\label{covariance}
B^k_{g_f} =e^{-kf}B^k_g\quad \text{on }\Sigma,\quad k=1,2,3.
\end{equation}
Define the bilinear form
\begin{equation*}
\mathcal Q_g(u_1,u_2)=\int_M u_1 P_g u_2\,dv_g+\oint_\Sigma \left(u_1 B_g^3(u_2)+B_g^1(u_1)B^2_g(u_2)\right)\,d\sigma_h
\end{equation*}
for $u_1,u_2\in\mathcal C^{\infty}(M)$. The main theorem in \cite{Case:boundary-operators} shows that $\mathcal Q_g$ is symmetric. The corresponding energy functional
\begin{equation*}\label{energy}
\mathcal E[u]=\mathcal Q_g(u,u)
\end{equation*}
 is a conformal invariant. Indeed,
\begin{equation}\label{conformal-invariance}
\mathcal E_{g_f}[u]=\mathcal E_g[u].
\end{equation}

The boundary operator $B_g^3$ operator is associated to the following curvature quantity
\begin{equation}\label{T}
T_g=\eta J-\frac{2}{3}\tilde \Delta H-2\langle \second_0,\tilde A \rangle+\frac{4}{3}H\tilde J+\frac{1}{3}H|\second_0|^2-\frac{2}{27}H^3.
\end{equation}
For a conformal metric $g_f=e^{2f}g$, the $T$-curvature equation is
\begin{equation*}
B^3_{g} f+T_g=T_{g_f} e^{3f}.
\end{equation*}
In addition, the integral quantity
\begin{equation*}
\kappa_{g,h}:=\int_M Q_g\,dv_g+\int_\Sigma T_g\,d\sigma_h
\end{equation*}
is a conformal invariant. \\

It is well known that the mean curvature is the associated boundary curvature to the scalar curvature on $M$, and that the pair $(R,H)$ is conformally covariant. From the PDE point of view, these arise from a boundary value problem for the conformal Laplacian (see \eqref{LN} below). If one considers instead fourth-order equations on manifolds with boundary, the couple $(Q,T)$ is the natural generalization of the pair $(R,H)$, and has been well studied: for the construction of constant $Q$-curvature metrics with vanishing $T$-curvature, see \cite{Nidiaye:constant-Q}, while the constant $T$-curvature problem was considered in \cite{Ndiaye:constant-T}. A $Q$-curvature flow on manifolds with boundary was analyzed in \cite{Ndiaye:flow}. In the particular case of  $(\mathbb B^4, \mathbb S^3)$ sharp Sobolev trace inequalities for the curvature $T$ were proved in
\cite{Ache-Chang}.\\

In addition, the pair $(Q,T)$ controls topology in the 4-dimensional setting. More precisely, there is a Chern-Gauss-Bonnet formula analogous to \eqref{Gauss-Bonnet} for 4-manifolds with boundary \cite{Chang-Qing:zeta1}:
\begin{equation}\label{CGB-full}
8\pi^2\chi(M)=\int_M \Big(\frac{|W|_g^2}{4}+Q_g\Big)\,dv_g+\int_{\Sigma} \Big(T_g-\frac{2}{3}\trace \second_0^3\Big)\,d\sigma_h.
\end{equation}
If $M$ is a l.c.f. manifold with umbilic boundary, this formula greatly simplifies:
\begin{equation}\label{CGB}
8\pi^2\chi(M)=\int_M Q_g\,dv_g+\int_{\Sigma} T_g\,d\sigma_h.
\end{equation}

Our first result in the boundary case is a classification statement based on the injectivity of the developing map $\Phi:M\to\mathbb S^4$ for a l.c.f. manifold, thus partially generalizing the seminal work by Schoen-Yau \cite{Schoen-Yau:paper}, \cite[Chapter VI]{Schoen-Yau:libro} to manifolds with boundary. We observe that the umbilicity assumption in the Theorem is a natural one, since it is a conformal invariant property.

We denote by  $\mathcal Y[g]$ Yamabe invariant for manifolds with boundary (see  \eqref{Y-invariant} for its precise expression). It is the natural generalization of \eqref{Yamabe-invariant-closed}, thus with a slight abuse of notation we  denote it by the same letter.

\begin{theorem}\label{thm:positivity}
Let $M$ be a compact,  orientable, l.c.f. Riemannian 4-manifold with umbilical boundary $\Sigma=\partial M$.
 Then,
 \begin{itemize}
 \item[a.]
  If $M$ is simply connected and $\Sigma$ has one connected component, then $M$ is conformally equivalent to a half-sphere
  $$\mathbb S^4_+=\{(z_0,\ldots,z_4)\in\mathbb R^{5}\,:\, |z|=1,z_0\geq 0\}.$$

\item[b.] If $M$ is not necessarily simply connected, but $\chi(M)=1$ and $\mathcal Y[g]>0$, then the same conclusion holds \cite{Raulot}.

\item[c.] \label{2boundaries} Assume that $M$ is simply connected, $\Sigma$ has exactly two connected components, $R_g > 0$ and $Q_g>0$. Then $M$ is conformally equivalent to an annulus in $\mathbb R^4$, that can be chosen as
    $$\mathcal A_\rho:=\{x\in\mathbb R^{4}\,:\,\rho\leq|x|\leq 1\} \quad\text{for some }\rho\in(0,1).$$
\end{itemize}
\end{theorem}

 We remark that statement a. of the previous theorem follows from a classical doubling argument and it has been already studied in the literature. On the other hand,
  statement b. above was proved by Raulot \cite{Raulot} and we include it here for completeness. Thus our main contribution is statement c. for the annulus case, which is  partly inspired in the work of Chang, Hang and Yang \cite{Chang-Hang-Yang} for closed manifolds of positive $Q$-curvature.

 Part a. may be understood as a 4-dimensional version of the classical Riemann mapping theorem in the plane. For the multiply connected case, part c. tells us that we cannot map two double-connected regions $M$ and $M'$  unless they share the same $\rho$. This is a very similar behavior to what happens in two dimensions, since two ring regions in the plane can only be mapped to one another unless they have the same \emph{extremal distance} or \emph{conformal modulus}, which is a conformal invariant quantity. This notion goes back to Ahlfors \cite{Ahlfors:book} (see also, for instance, the more modern exposition of \cite{Lawler}).\\

The proof of Theorem \ref{thm:positivity} also relies in the study of the developing map. A conformally invariant quantity, that will be relevant in analyzing this developing map was defined by Escobar and it is the analogue of the Yamabe invariant for manifolds with boundary.   This invariant is crucial in the so-called Yamabe problem with boundary, which seeks a conformal metric on $M$ to a given one that has constant scalar curvature and zero mean curvature on the boundary. The Yamabe problem with boundary was solved in many cases by Escobar in \cite{Escobar:Yamabe-with-boundary} (in particular, in dimension four which suffices for our purposes). Related work can be found in \cite{Cherrier:problemes, Han-Li:Yamabe-with-boundary,Ambrosetti-Li-Malchiodi,Mayer-Ndiaye:remaining-cases,Brendle-Chen} from the variational point of view, and \cite{Brendle:Yamabe-flow,Almaraz-Sun:flow} for flow-type methods.

Note in addition that  since the right hand side of \eqref{CGB} is conformally invariant, it is convenient to take Escobar's solution as a background metric in $M$ and in this particular case, $T\equiv 0$.\\

Our final goal in this paper is to understand the properties and eigenvalues  of the third-order boundary operator $B_g^3$. To this operator we need to associate a second boundary condition, so we will work on the class of functions
\begin{equation*}
\begin{split}
\mathcal U_0=\{u:\overline M\to\mathbb R \,:\, u \text{ smooth}, \,\partial_\eta u=0\text{ on }\Sigma\}.
\end{split}
\end{equation*}
In this class the energy functional reduces to
\begin{equation}\label{energy1}
\begin{split}
\mathcal E_g^M[u]
&=
\int_{M} (\Delta_g u)^2\,dv_g+\int_M \Big(2Jg_{ab}-4A^g_{ab}\Big)\nabla^a u\nabla^b u\,dv_g+\frac{2}{3}\int_\Sigma H|\tilde\nabla u|_h^2\,d\sigma_h-2\int_\Sigma (\second_0)_{ij}\tilde\nabla^i u\tilde\nabla^j u\,d\sigma_h.
\end{split}
\end{equation}
Thus we would like to study the boundary eigenvalue problem
\begin{align}
P_gu &=0 \hbox{ in } M, \label{bi-laplacian-eq}\\
B^1_g u&=0  \hbox{ on } \Sigma, \label{boundary-condition}\\
B^3_g u&= \lambda u  \hbox{ on } \Sigma. \label{eigenvalue-problem}
\end{align}
It is possible to show that there exists an increasing sequence of eigenvalues
\begin{equation*}
\lambda_0^g\leq \lambda_1^g\leq \lambda_2^g\leq\ldots
\end{equation*}

A straightforward calculation from the models yields a statement about the positivity of $B^3_g$:

\begin{corollary}\label{cor:positivity}
In all cases a., b., c. in Theorem \ref{thm:positivity} above we have $\lambda_0^g=0$ and the corresponding eigenspace consists only of constant functions.
\end{corollary}

This implies, in particular, that $\lambda_1^g>0$ may be characterized by the following Rayleigh-type quotient:
\begin{equation}\label{Rayleigh-quotient1}
\lambda_1^g=\min_{\mathcal U_0\,:\,\int_\Sigma u=0}\frac{\mathcal E_g^M[u]}{\displaystyle\int_\Sigma u^2\,d\sigma_h}.
\end{equation}
The question of positivity of $B^3_g$ has also been analyzed in other contexts, see for example the work in \cite{Case-Chang}.\\

Now we look at the min-max problem for $\lambda_1^g$. Our main Theorem is the four-dimensional generalization of \eqref{Steklov}, which may be applied to manifolds satisfying the hypothesis of Theorem \ref{thm:positivity}:

\begin{theorem} \label{thm:bounds-boundary}
We get the following bounds for $\lambda_1^g>0$:
\begin{itemize}
\item[\emph{a.}] If $M$ is conformally equivalent to a half-sphere $\mathbb S^4_+$,
    \begin{equation*}
    \lambda_1^g \vol(\Sigma)\leq 24\pi^2.
    \end{equation*}
    and it is attained at a flat disk.
\item[\emph{b.}] If $M$ is conformally equivalent to an annulus $\mathcal A_\rho$ (with boundaries $\Sigma_1$, $\Sigma_ \rho$),
\begin{equation}\label{statement-annulus}
    \lambda_1^g \vol(\Sigma)\leq c\Big(\rho,\frac{\vol(\Sigma_\rho)}{\vol(\Sigma_1)}\Big),
    \end{equation}
    where this is a constant can be explicitly computed.

    In addition, there is $\rho^*>0 $ such that for $\rho\leq \rho^*$ the bound is sharp.
\end{itemize}
\end{theorem}

The bounds of the previous theorem are obtained by comparison with explicit computations in two types of model manifolds: a 4-dimensional ball and  4-dimensional annuli (see Section \ref{section:calculations}). The computations in the ball model are a   are straightforward and provide optimal bounds. On the other hand, the precise  calculations for the annuli are based on the ideas in
 \cite{Fraser-Schoen:annulus}, and although elementary, they soon become quite technical.\\

Finally, we make some bibliographical remarks on related eigenvalue problems. For a general introduction to boundary value problems for fourth-order operators we refer to the monograph \cite{Gazzola-Grunau-Sweers}. Many versions of (fourth-order) eigenvalue problems in which the eigenvalue appears in the boundary condition have appeared in the literature \cite{Buoso-Provenzano:biharmonic-Steklov,Bucur-Ferrerro-Gazzola,Liu:Weyl,Gazzola-Sweers,Bucur-Gazzola,Knopf-Liu}. These are known as biharmonic Steklov eigenvalue problems. One particular application of this is to study suitable boundary conditions for the Cahn-Hilliard equation. This is a model that describes phase separation
processes of binary mixtures by a non-linear fourth-order equation. In recent years, several types of dynamic
boundary conditions have been proposed in order to account for the
interactions of the material with the solid wall and, in particular, third order boundary conditions play an essential role in the model \cite{Liu-Wu:Cahn-Hilliard,Knopf-Lam-Liu-Metzger}.\\

The organization of the paper is as follows: In Section \ref{section:closed} we discuss the eigenvalue problem for the Paneitz operator on closed manifolds and we prove Theorem \ref{thm:closed}. In Section \ref{section:Escobar} we discuss preliminary notions that are necessary  in the proof
the classification Theorem  \ref{thm:positivity}, that is  finally proved in Section \ref {section:injectivity}. Particular geometric models are analyzed in more detail in Section \ref{section:calculations} and in the Appendix. We finally prove  Corollary \ref{cor:positivity} and Theorem \ref{thm:bounds-boundary}
 in Section \ref{sec: boundary value}.\\

\noindent{\it Acknowledgments: The authors would like to thank Jeffrey Case, Alice Chang, Gaven Martin, Vicente Mun\~{o}z and Riccardo Piergallini, for many useful discussions and suggestions.

The authors would like to also thank the anonymous referee that pointed out a gap in the first version of this manuscript and led to improvements of our work.}



\section{The closed case}\label{section:closed}

In this section we give the proof of Theorem \ref{thm:closed}, which we summarize here: first, the positive curvature assumption allow us to control the topology of $M$ and, either $M$ is conformally equivalent to a sphere $\mathbb S^4$, or $M$ is covered by a cylinder $\mathbb R\times \mathbb S^3$. In the first case, we can obtain an upper bound for $\lambda_1$ using the scheme in Yang-Yau \cite{Yang-Yau} , which is based on a trick by Hersch \cite{Hersch}.  Hersch's idea is to use the coordinate functions of the embedding as test functions in the Rayleigh quotient \eqref{Rayleigh}. A modification of this strategy yields the cylinder case too.\\

We recall now some facts about locally conformally flat (l.c.f.) manifolds; for additional  background, we refer to the book \cite[Chapter VI]{Schoen-Yau:libro}.  A Riemannian metric $g$ on a smooth manifold $M$ is called l.c.f. if for every point $p\in M$, there exists a neighbourhood $U$ of
$p$ and a smooth function $f$ on $U$  such that the metric $e^{2f}g$ is flat on $U$.
Note that, in dimension $4$, a Riemannian manifold is locally conformally flat if and only if the Weyl tensor $W$ vanishes.\\

We assume, to start with, that $M$ is a simply connected, closed, compact, l.c.f. manifold of dimension $n$. Liouville's theorem \cite[Theorem 1.6 in Chapter VI]{Schoen-Yau:libro} allows us to patch all these neighborhoods to obtain a globally defined conformal immersion $\Phi:M\to \mathbb R^n$ (or equivalently, $\Phi:M\to \mathbb S^n$ by stereographic projection), such that  the locally conformally flat structure of $M$ is induced by $\Phi$. The function  $\Phi$  is called the \emph{developing map} and it is unique up to conformal transformations of $\mathbb S^n$.\\

Note that a simple topological argument yields the well known characterization result by Kuiper \cite{Kuiper} (see also the notes \cite{Howard} for remarks on regularity). Indeed, $\Phi(M)$ is at the same time open and closed in $\mathbb S^n$. More precisely:

\begin{theorem}[Kuiper \cite{Kuiper}]
Any $n$-dimensional closed simply-connected locally conformally flat manifold is conformally equivalent to $\mathbb S^n$.
\end{theorem}


\begin{proof}[Proof of Theorem \ref{thm:closed}.i.]
By our previous discussion, there is a bijective conformal embedding $\Phi: (M^4,  g)\to (\mathbb{S}^4, g_{\mathbb{S}^4})$, where $g_{\mathbb S^4}$ is the canonical metric on the sphere. We denote by $(z_0,z_1,\ldots,z_4)$ the coordinates of $\mathbb S^4$ in $\mathbb R^5$ and by $\Phi_i$ the $i$-th coordinate of the embedding $\Phi$, $i=0,1,2,3,4$.\\

In this setting, we must have $\lambda_1^g>0$ with $\Ker(P_g)=\{\text{constants}\}$ since condition \eqref{Gursky-condition} is trivially satisfied on the sphere.\\

Let us check now that $\Phi_i$ is an admissible function for the Rayleigh quotient \eqref{Rayleigh}.
A standard calibration argument (see Lemma 1.1 in \cite{Hersch} or page 107 in \cite{Gromov}) yields that we can choose the embedding satisfying
 \begin{equation}\label{calibration}
\int_M \Phi_i\,dv_g=0,\quad i=0,\ldots,4.
\end{equation}
Moreover,
 \begin{equation}\label{test-function}
 \lambda_1^g \int_M \Phi_i^2\,dv_g\leq  \mathcal E^M_g[\Phi_i]=  \mathcal E^{\Phi(M)}_{ g_{\mathbb{S}^4}}[z_i],
 \end{equation}
Adding on $i$ we have
$$ \lambda_1^g\vol(M)\leq \sum_{i=0}^4 \mathcal E^{\Phi(M)}_{ g_{\mathbb{S}^4}}[z_i],$$
here we have used that $\sum_{i=0}^4\Phi_i^2=1$. Now recall that $\Phi:M\to \mathbb S^4$ is  bijection and calculate, from the expression of the energy \eqref{Eg},
$$\mathcal E^{\mathbb{S}^4}_{g_{\mathbb{S}^4}}[z_i]=\mu_1^2\int_{\mathbb{S}^4} z_i^2\,dv_{g_{\mathbb{S}^4}}+ 2\int_{\mathbb{S}^4} |\nabla z_i|^2 \,dv_{\mathbb{S}^4}=
\big(\mu_1^2+2\mu_1\big)\int_{\mathbb{S}^4} z_i^2\,dv_{g_{\mathbb{S}^4}}.$$
Here $\mu_1=4$ is the first non-zero eigenvalue of the (minus) Laplace-Beltrami operator on $\mathbb{S}^4$.
 Thus we conclude
$$\sum \mathcal E^{\Phi(M)}_{ g_{\mathbb{S}^4}}[z_i]=  8\vol(\mathbb{S}^4)=64 \pi^2,$$
and that this bound is sharp, since the coordinate functions are already eigenfunctions. This completes the proof.
\end{proof}

We next consider the non-simply connected case; in the l.c.f. 4-dimensional setting it turns out that positive curvature gives information about the topology. Since the Weyl term vanishes for l.c.f manifolds, under the assumption $\kappa_g\geq 0$,  the Gauss-Bonnet formula  \eqref{Gauss-Bonnet}   implies that $\chi(M)\geq 0$.
The classification of such manifolds according to the Euler characteristic was studied by Gursky in  \cite[Theorem A]{Gursky:lcf-4-6}:  if $M$ is a compact 4- or 6-dimensional manifold which admits a l.c.f. of non-negative scalar curvature $g$, then $\chi(M) \leq 2$. Furthermore, $\chi(M) = 2$ if and only if $(M,g)$ is conformally equivalent to the sphere with its canonical metric, and $\chi(M) = 1$ if and only if $(M,g)$ is conformally equivalent to projective space with its canonical metric.
The remaining case $\chi(M)=0$ was characterized in \cite[Corollary G]{Gursky:Weyl}: if $(M,g)$ is a compact, l.c.f. 4-manifold with $Y[g]>0$ and $\chi(M) = 0$, then $(M,g)$ is conformal to a quotient of the cylinder $\mathbb R\times \mathbb S^3$.\\

A related result was proven by Chang, Hang and Yang \cite[Corollary 1.2]{Chang-Hang-Yang}. More precisely, if $M$ has positive scalar curvature and positive $Q$-curvature, then $M$ is conformally equivalent to a quotient of the sphere. Note that if we remove the positive $Q$-curvature assumption one may construct examples of manifolds that are conformally equivalent to $\mathbb S^4\setminus\{p_1,\ldots,p_N\}$ (see \cite[Theorem 1.3]{Chang-Hang-Yang} and the discussion there).

As a side remark, closed, flat manifolds are isometric to $\mathbb R^n / \Gamma$, for $\Gamma$ a Bieberbach group.
A short overview on Bieberbach manifolds can be found  in \cite[Section 4.1]{Bettiol-Piccione}.
\\

\begin{proof}[Proof of Theorem \ref{thm:closed}.ii.]
It follows as part \emph{i.} taking into account that $M$ is orientable.
\end{proof}

Now we deal with the remaining case in which $M$ is (conformally) covered by a cylinder $\mathbb R\times \mathbb S^3$. These manifolds have been studied in \cite{Hillman:cylinder}, \cite[Chapter 11]{Hillman:book}. It is known that a closed 4-manifold $M$ is covered by $\mathbb R\times\mathbb S^3$ if and only if $\pi_1 = \pi_1(M)$ has two ends and $\chi(M) = 0$. Its homotopy type is then determined by $\pi_1$ and the first nonzero $k$-invariant $k(M)$. While  all the possible subgroups  of $\pi_1( \mathbb R\times\mathbb S^3) $ are well known,  there is not a complete classification of which manifolds can be actually realized with such fundamental groups (see also \cite{Hamilton} for examples of quotients with positive curvature). In any case, we assume that  the fundamental domain $\Omega:=\Phi(M)$ is exactly a region $[0,\varrho)\times \mathbb S^3$ of $\mathbb R\times\mathbb S^3$ for some $\varrho>0$, with periodicity in the real variable.
\begin{proof}[Proof of Theorem \ref{thm:closed}.iii.]
Note first that condition \eqref{Gursky-condition} for positivity of $\lambda_1^g$  is trivially satisfied by our hypothesis.

To construct suitable test functions we consider the coordinates on the sphere $\mathbb S^3$ and use a variation of Hersch's calibration method in \cite{Hersch}. More precisely, if $(t, y_1, y_2, y_3, y_4)\in [0,\varrho)\times\mathbb S^3$ we take $(y_1, y_2, y_3, y_4)\in \mathbb{S}^3$, and apply a Moebius transformation $\varphi_{p,\delta}$ of $\mathbb S^3$ which is induced by dilations on the tangent plane at a point $p\in\mathbb S^3$. We briefly sketch Hersch's argument to show that by his procedure we can  find  $p\in \mathbb{S}^3$ and $\delta\in (0,1]$ such that
\begin{equation}\label{calib}\int_M x_i\circ \varphi_{p,\delta}\circ \Psi_{\mathbb{S}^3} \,dv_g=0 \quad\text{for}\quad i\in\{1,2, 3, 4\}. \end{equation}
Here $x_i$ is the $i$-th coordinate function on $\mathbb{S}^3$ and $\Psi_{\mathbb{S}^3}$ the projection onto $\mathbb{S}^3$ of the  conformal embedding of $M$ into $[0,\varrho) \times \mathbb S^3$. Let
$$F_i(p,\delta)= \int_M x_i\circ \varphi_{p,\delta}\circ \Psi_{\mathbb{S}^3}\, dv_g$$ and $F=(F_1, F_2, F_3, F_4)\in \mathbb{R}^4 $.
Note first that   $F\to p \vol(M)$ as $\delta\to 0$ and, in particular, the surface $F(\cdot,\delta)$ tends a sphere of radius $\vol(M)$ that does not touch the origin. On the other hand,  $F(p,1)$ is a fixed point independent of $p$, while for any given $\delta\in (0,1)$, $F(\cdot,\delta)$  is an immersed 3-dimensional surface on $\mathbb{R}^4$ that is continuous in $\delta$. If $F(p,1)$ is 0 we already have the desired transformation, otherwise, a continuity argument in $\delta$ and $p$ implies that there are $p\in \mathbb{S}^3$ and $\delta\in (0,1)$ such that $F(p,\delta)$ agrees with the origin, which is precisely \eqref{calib}. Note, in addition, that
$$\sum_{i=1}^4(x_i\circ\varphi_{p,\delta}\circ \Psi_{\mathbb{S}^3})^2=1,$$
which yields
$$ \lambda_1^g\vol(M)\leq \sum_{i=1}^4 \mathcal E^{\Phi(M)}_{ g_{\mathbb R\times\mathbb{S}^3}}[x_i\circ\varphi_{p,\delta}],$$
here we  used that the energy is a conformal invariant. To complete the proof we explicitly compute the energy of these transformations and maximize our computation in $\delta\in [0,1]$. By symmetry, it is enough to consider the Moebius  transformations with $p=S$ the  the South pole, which are given by  $$\varphi_{p, \delta}(\hat{y}, y_4)=\left( \frac{2 \delta\hat{y}}{(1-y_4)+\delta^2 (1+y_4)}, \frac{y_4-1+\delta^2 (1+y_4)}{(1-y_4)+\delta^2 (1+y_4)} \right),$$
where we write $\hat y=(y_1,y_2,y_3)$.
Then, denoting by $\varphi^i$ the $i$-th coordinate of $\varphi_{p, \delta}(\hat{y}, y_4)$ (or equivalently, $x_i\circ \varphi_{p, \delta}$) and $f(y_4, \delta)= \frac{1}{(1-y_4)+\delta^2 (1+y_4)}$, the energy is
  given by
\begin{equation*}
\begin{split}\sum_{i=1}^4 E^{\Phi(M)}_{ g_{[0,\varrho)\times\mathbb{S}^3}}[x_i\circ\varphi_{p,\delta}]=\int_{[0,\rho)\times \mathbb S^3}
4\delta^2 f^2 |\hat{y}|^2 \big[&3+5y_4(1-\delta^2)f-2 (1-\delta^2)^2(1-y_4^2)f^2]^2\\
&+ 16 \delta ^4 f^4[3y_4- 2(1-y_4^2)(1-\delta^2) f\big]^2\, d\mu_M.
\end{split}
\end{equation*}
To obtain a uniform bound in $\delta$ for this energy, we parametrize $\mathbb S^3$ by $(\sin \phi \, \, \omega, \cos \phi)$, where $\omega\in \mathbb S^2$ and $\phi\in[0,\pi)$. Then the volume element is given by $\sin^2\phi \, \mu_{\mathbb{S}^2}$, where $\mu_{\mathbb{S}^2}$ is the volume element of $\mathbb{S}^2$,
\begin{align*}
\sum_{i=1}^4 E^{\Phi(M)}_{ g_{[0,\varrho)\times\mathbb{S}^3}}[x_i\circ\varphi_{p,\delta}]&=16\delta ^2 \pi\varrho\int_0^\pi
 f^2 \sin^4\phi\,\big[3+5\cos \phi(1-\delta^2)f-2 (1-\delta^2)^2f^2 \sin^2\phi\big]^2 \,d\phi\\ &+
 64\delta ^4 \pi\varrho\int_0^\pi
 f^4\sin^2 \phi \,\big[3\cos \phi- 2 \sin^2\phi(1-\delta^2) f\big]^2\, d\phi,
\end{align*}
and $f(\phi, \delta)=\frac{1}{1-\cos \phi+\delta^2(1+\cos\phi)}.$
As $\delta \to 1,$ it is easy to verify that
\begin{equation*}f\to \frac{1}{2}\qquad \text{and}
\qquad \sum_{i=1}^4 E^{\Phi(M)}_{ g_{[0,\varrho)\times\mathbb{S}^3}}[x_i\circ\varphi_{p,\delta}]\to
18 \pi^2 \varrho.
\end{equation*}
To study the behavior for $\delta$ small we  observe that %
$$ E^{\Phi(M)}_{ g_{[0,\varrho)\times\mathbb{S}^3}}[x_i\circ\varphi_{p,\delta}]\leq C \varrho \pi \int_0^\pi f^4 \sin^4\phi \,d\phi.$$
Since $f$ is decreasing in $\delta,$ we have that $f(\cos \phi, \delta)\leq \frac{1}{1-\cos \phi}$ and $f\sin^2\phi \leq (1+\cos\phi)\leq 2.$ Then,
$$ E^{\Phi(M)}_{ g_{[0,\varrho)\times\mathbb{S}^3}}[x_i\circ\varphi_{p,\delta}]\leq C \delta^2 \varrho \pi \int_0^\pi f^2  \,d\phi \leq C \frac{\varrho \pi}{\delta} .$$
The constants and the full energy can be explicitly computed, but we avoid it here for simplicity.\\

Now we use Condition \eqref{concentrating} to find a lower bound for  $\delta$. Again, for simplicity we assume that the Moebius transformation is centered at the South pole $p=S$. With a slight abuse of notation we identify   $\mathcal B_\delta(N)=\Psi^{-1}(\mathcal B_\delta(N))$, where $N$ is the North pole.

Then \eqref{calib} yields
$$\int_{M\cap \mathcal B^c_\delta(N)} x_4\circ \varphi_{p,\delta}\circ \Psi_{\mathbb{S}^3} \,dv_g=- \int_{M\cap \mathcal B_\delta(N)} x_4\circ \varphi_{p,\delta}\circ \Psi_{\mathbb{S}^3} \,dv_g.  $$
Observe that $y=(y_1,\, y_2, \, y_3, \, y_4) \in \mathcal B_\delta(N)$ implies that $0\leq 1-y_4 < C\delta$. Then for $(y_1,\, y_2, \, y_3, \, y_4)\in \mathcal B^c_\delta(N)$ it holds that   $0\leq y_4\circ\varphi_{p, \, \delta}(y) +1 \leq C \delta,$  and
\begin{align*} - \int_{M\cap  \mathcal B_\delta(N)} x_4\circ \varphi_{p,\delta}\circ \Psi_{\mathbb{S}^3} \,dv_g & =  \int_{M\cap \mathcal B^c_\delta(N)} (1+x_4\circ \varphi_{p,\delta}\circ \Psi_{\mathbb{S}^3}) \,dv_g
 -\vol_M(M\cap \mathcal B^c_\delta(N))\\
& \leq  (C\delta -1)\vol_M(M\cap \mathcal B^c_\delta(N)). \end{align*}
 Since $|x_4\circ \varphi_{p,\delta}\circ \Psi_{\mathbb{S}^3}|\leq 1,$
we conclude that
$$-\vol_M(M\cap \mathcal B_\delta(N))
 \leq  (C\delta -1) \vol_M(M\cap \mathcal B^c_\delta(N)). $$
Using  condition \eqref{concentrating} we have that if $\delta<\delta_0$ then
$$1-\varepsilon\leq 1-\frac{\vol_M(M\cap  \mathcal B_\delta(N)) }{\vol_M(M\cap \mathcal B^c_\delta(N))}\leq \delta.$$
This concludes the proof.\\

Finally, we point out that test functions need to be periodic in the $t\in [0,\varrho)$ variable since $[0,\varrho)\times \mathbb{S}^3$ is the fundamental domain of a  quotient. Note, in addition, that the transformations $\varphi_{p,\delta}$ are  periodic in $t$, but not conformal on the cylinder. Nonetheless, they  provide suitable test functions for which the energy can be explicitly computed.
\end{proof}

 In Lemma \ref{lemma:periodic} we will calculate the precise eigenvalue for the canonical metric in $[0,\varrho)\times \mathbb S^3$, which shows, on the one hand, that our bound is far  to be sharp when $\varrho\to \infty$ and, on the other hand, justifies the need of some geometric condition such as \eqref{concentrating} when $\varrho\to 0$. Indeed, the lowest positive eigenvalue is
\begin{equation*}
    \big[(2+(\tfrac{2\pi}{\varrho})^2\big]^2 -4.
\end{equation*}

\section{Preliminaries on the boundary case}\label{section:Escobar}

\subsubsection{Escobar's problem}
Here we recall some background on the Yamabe invariant for manifolds with boundary.  We use the notation from Subsection \ref{section:introduction-boundary} in the Introduction. Let  $(M,g)$ be a compact, $n$-dimensional, Riemannian manifold with boundary $\Sigma=\partial M$, and let $h$ be the restriction of the metric $g$ to the boundary. The first observation is that  the conformal Laplacian on $M$ can be associated to a boundary operator $N_g$  on $\Sigma$. We set
\begin{equation}\label{LN}
\left\{\begin{split}
&L_g u:=-\Delta_g u+\tfrac{n-2}{4(n-1)}u\quad \text{in }M,\\
&N_g u:=\partial_\eta u+\tfrac{n-2}{2}H_g u\quad \text{on }\Sigma.
\end{split}\right.
\end{equation}
We note that $N$ plays the role of a Neumann (more precisely, Robin) condition. The most important property for this system is that the couple $(L,N)$ is conformally covariant. Indeed, for a conformal change $g_u=u^\frac{4}{n-2}g$ we have
\begin{equation}\label{eq-conformal}
\begin{split}
&L_{g_u}(u^{-1}\phi)=u^{-\frac{n+2}{n-2}}L_ {g}\phi\quad\text{in }M,\\
&N_{g_u}(u^{-1}\phi)=u^{-\frac{n}{n-2}}N_{g}\phi \quad\text{on }\Sigma.
\end{split}
\end{equation}
The Yamabe problem for manifolds with boundary asks to find a conformal metric to $g$ with constant scalar curvature on $M$ and zero mean curvature on $\Sigma$. In PDE language we look for a positive solution to
\begin{equation}\label{problem-Escobar}
\left\{\begin{split}
&L_g u=cu^{\frac{n+2}{n-2}}\quad \text{in }M,\\
&N_g u=0\quad \text{on }\Sigma.
\end{split}\right.
\end{equation}
This problem was first studied by Escobar \cite{Escobar:Yamabe-with-boundary}. He solved it in many cases, including the 4-dimensional, l.c.f., umbilic boundary case which is the setting of this paper. More precisely, a solution  may be found using variational methods for the following  Yamabe invariant
\begin{equation}\label{Y-invariant}\mathcal Y[g]=\inf\{\mathcal R_g[u]\,:\, u\in W^{1,2}(M),\, u\not\equiv 0\}\end{equation}
where
\begin{equation*}
\mathcal R_g[u]=
\frac{\displaystyle\int_M u L_gu\,dv_g}{\displaystyle\Big(\int_M u^{\frac{2n}{n-2}}\,dv_g\Big)^{\frac{n-2}{n}}}
=\frac{\displaystyle\int_M \left(|\nabla u|^2_g+\tfrac{n-2}{4(n-1)} R_g u^2\right)\,dv_g+\tfrac{n-2}{2}\int_\Sigma H_g u^2\,d\sigma_h}{\displaystyle\Big(\int_M u^{\frac{2n}{n-2}}\,dv_g\Big)^{\frac{n-2}{n}}}.
\end{equation*}
It is well known that a (positive) solution exists if $\mathcal Y[g]<\mathcal Y[g_{\mathbb S^n_+}]$, and that if equality is attained then $M$ is already diffeomorphic to the model $\mathbb S^n_+$. In addition, the sign of $\mathcal Y[g]$ coincides with the sign of the first eigenvalue for the conformal Laplacian on $M$ (coupled with the boundary condition $N_gu=0$).

\begin{remark}
In the following, we may assume without loss of generality that our background metric $g$ has constant scalar curvature and zero mean curvature on the boundary.
\end{remark}

\subsubsection{ Raulot's Theorem}

In this subsection we recall the following theorem of Raulot~\cite{Raulot} in dimensions 4 and 6:

\begin{theorem}[\cite{Raulot}]\label{thm:Raulot} Let $n$ be 4 or 6, and let $M$ be an $n$-dimensional l.c.f. manifold with umbilical boundary. Assume that the Yamabe invariant satisfies $\mathcal Y[g]\geq 0$. Then the Euler characteristic satisfies the bound $\chi(M)\leq 1$. In addition, if $\chi(M)=1$ and $\mathcal Y[g]>0$, then $(M, g)$ is conformally equivalent to the standard half-sphere $\mathbb S^n_+$.
\end{theorem}

Note that this immediatly yields part b. of  Theorem \ref{thm:positivity}.

\section{Boundary case: Proof of Theorem \ref{thm:positivity}} \label{section:injectivity}

In this section we prove Theorem \ref{thm:positivity}. We recall that in the locally conformally flat setting  there exists a conformal map $\Phi:M\to\mathbb S^4$ (the developing map). The key ingredient of our result is the injectivity of such map $\Phi$ under hypotheses of  Thereom \ref{thm:positivity}.  

\subsection{One boundary component.}\label{subsection:one-component} We begin by proving the following result.

\begin{proposition}\label{prop:injectivityi}
Let $(M,g)$ be a simply connected, compact, l.c.f. 4-dimensional Riemannian manifold with umbilic boundary $\Sigma$. Assume that
 $\Sigma$ has one  connected component, then the developing map $\Phi:M\to\mathbb S^4$ is injective and thus, a diffeomorphism onto its image.
\end{proposition}

\begin{proof}

In this situation the injectivity of the developing map relies on a classical doubling argument. We follow the presentation of Spiegel \cite{Spiegel} to describe the construction.

Consider the doubling of the manifold $M$ defined as $\hat M = M \cup (-M)$, where  we write $-M$ for a
second copy of $M$ that is  distinguished  from $M$ itself (for instance by taking $M\times\{1\} $ and $M\times\{-1\} $). Here the manifold $M$ and its copy $-M$ are identified at their boundaries $\Sigma$ and hence $\hat M$ is a closed manifold (see \cite[Chapter 5]{book:doubling} for more details).
Since $\Sigma=\partial M \subset \hat M$ is umbilic, and this is a conformal invariant property, the image of $\Sigma$
must be umbilic in $\mathbb S^4$, thus it must be  contained in a hypersphere of $\mathbb S^4$. Now, since $\Phi|_{\Sigma}$ is a local diffeomerphism from a compact manifold to a simply connected manifold, we have that $\Phi|_{\Sigma}$ is actually diffeomorphism.
Composing
with a M\"obius transformation of $\mathbb S^4$, if necessary, we may assume that $\Phi(\Sigma)$ is the
equator $\{z=(z_0,\ldots,z_{4}) \in\mathbb S^4\subset \mathbb R^{5} \,;\, z_{0} = 0\}$.

Now take the odd extension of $\Phi$ to $\hat M$,
as follows:
\begin{equation*}
\hat \Phi(p):=(\hat \Phi_0(p),\ldots,\hat \Phi_{4}(p)),
\end{equation*}
where $\hat \Phi_i(p)=\Phi_i(p)$ for $i=1,\ldots,4$ and
\begin{equation*}
\hat \Phi_{0}(p)=\left\{
\begin{split}
&\Phi_{0}(p) \,\text{ if }p\in M,\\
&-\Phi_{0}(p) \,\text{ if }p\in -M.
\end{split}\right.
\end{equation*}
Now, by a straightforward connectedness argument again, we conclude that $\hat \Phi:\hat M\to \mathbb S^4$ is a diffeomorphism and thus, $\Phi$ is injective, as desired.

\end{proof}

\begin{proof}[Proof of Theorem \ref{thm:positivity}.a.]

From the proof above we have  (perhaps after a Moebius transformation) that $\Phi(\Sigma)$ is an equator. Since $\Phi:\hat M\to \mathbb{S}^4$ is a bijective diffeormphism, then necessarily the restriction of $\Phi$ to $M$ maps the manifold diffeomorphically into a hemisphere $\mathbb{S}^4_+$.

\end{proof}

\subsection{Two boundary components.}

We now study the case with two boundaries.  As explained in Section \ref{section:Escobar},  we may choose a conformal metric on $M$ such that the scalar curvature $R$ is constant in $M$ and the boundary is umbilic and minimal. Then,
we can again  consider again the doubling $\hat M$  of the manifold $M$ (following the same construction of  Subsection \ref{subsection:one-component}) and using the result in the Appendix
of \cite{Escobar:Yamabe-with-boundary}  we have that  $\hat M$ is smooth with Escobar's metric (and actually has a well characterized Green's function).  Moreover,
since $\hat M$ is a compact l.c.f. Riemannian manifold without boundary, the developing map $\Phi:\hat M\to\mathbb S^4$ exists.

We remark that the metric in the doubling $\hat M$, denoted by $\hat g$, can be taken to be $C^{2,\alpha}$ smooth. In addition, $R_{\hat g}>0$. Thus we can apply the results in \cite[Theorem 4.1]{Schoen-Yau:libro} to conclude that $\hat M$ is conformally equivalent to a quotient $\Omega/\Gamma$ for some domain $\Omega$ in $\mathbb S^n$.

Now we restrict to the image of $M$ by the developing map, denoted by $\Omega'=\Phi(M)$. It is a subdomain in $\Omega$. The boundary of $\Omega'$ can be written as $\Gamma:=\Phi(\Sigma)\cup \mathcal B$, where the latter  is a set of branching points.

We first analyze the boundary $\Sigma=\partial M$, which we recall is umbilic and hence the image of $\Sigma$ must be umbilic in $\mathbb S^4$, that is, each component of $\Sigma$ must be contained in a hypersphere of $\mathbb S^4$. Now take into account that, for each connected component $\Sigma'$ of $\Sigma$, $\Phi|_{\Sigma'}$ is a local diffeomorphism from a compact manifold to a simply connected manifold, so we must have that $\Phi|_{\Sigma'}$ is actually diffeomorphism. This implies that one can find a small neighborhood around $\Sigma$ in $M$ such that $\Phi$ is actually a local diffeomorphism and hence, no branching points can  occur in this set.

To analyze $\mathcal B$ away from $\Phi(\Sigma)$, we consider $\Omega'$ with the metric induced by the original metric $g$ in $M$ (not Escobar's). Since with the metric $g$
we assumed that the $Q-$curvature is positive, we
can use the arguments in the proof of  \cite[Theorem 1.2]{Chang-Hang-Yang} to conclude that the set $\mathcal B$  is empty. It is important to observe that this is possible since the proof of \cite{Chang-Hang-Yang} is local around each point $x\in\mathcal B$ and we argued in the paragraph above that $\mathcal B$ is at a positive distance of $\Gamma$.

In summary, we conclude that $\Phi$ cannot have branching points and it is a local diffeomorphism.
This in particular implies that $ \hat M$ can be identified with a quotient of $\mathbb{S}^4$ and thus, by restricting the developing map $\Phi:\hat M\to \mathbb S^4$  to $M$ we have that the developing map of $M$ is injective.

We remark that, in fact, under our assumptions, the classical proof of injectivity for the developing map by Schoen-Yau  \cite{Schoen-Yau:libro} can be performed directly for  manifolds with boundary since all the additional boundary integral terms that appear would vanish.

\begin{proof}[Proof of  Theorem  \ref{thm:positivity}.c.]

Consider the developing map  $\Phi: \hat{M}\to \mathbb{S}^4$.  We have that $\Phi$ is a conformal diffeomorphism onto its image (regular at all points).
Recall again the boundaries are assumed to be umbilic and hence their images are umbilic in $\mathbb S^4$.
Now, since $\Phi:M\to \mathbb S^4$ is an immersion, three situations can occur:
\begin{itemize}
\item Both components of $\Sigma$ are mapped to the same great circle in $\mathbb S^4$.
\item Each component of $\Sigma$ is mapped to two different great circles in $\mathbb S^4$ with non-empty intersection.
\item Each component of $\Sigma$ is mapped to two different circles in $\mathbb S^4$ with empty intersection. Thus $\Phi(M)$ an annulus type region in $\mathbb S^4$ (or $\mathbb R^4$ by stereographic projection).
\end{itemize}
The first and second situations are ruled out by the injectivity of the developing map,
thus we conclude that we are the third case and this finishes the proof of Theorem  \ref{thm:positivity}.c.

\end{proof}

\section{Explicit calculations in known models}\label{section:calculations}

In this section we provide the explicit solution to the eigenvalue problem  \eqref{bi-laplacian-eq}-\eqref{eigenvalue-problem} for a family of rotationally symmetric metrics.

In the particular case that $M$ is a flat ball $\mathbb B^4_{r_1}$ of radius $r_1$ in $\mathbb R^4$, and $\Sigma=\partial M$ the a sphere $\mathbb S^3_{r_1}$ with its canonical metric, we can simply write:
\begin{align}
&P_0=(-\Delta)^2,\nonumber\\
&B^1_0=\partial_r,\nonumber\\
\label{Panetiz-boundary-ball}
&B^3_{0}=-\partial_r\Delta-2\tilde \Delta \partial_r -\frac{2}{r_1}\tilde \Delta  - \frac{1}{r_1^2}\partial_r.
\end{align}
However, for our purposes it is more convenient to rewrite these operators in cylindrical coordinates.

\subsection{Cylindrical coordinates}

First, we write the flat metric as
\begin{equation}\label{cylinder-metric}
|dx|^2=dr^2+r^2d\theta^2=e^{-2t}[dt^2+d\theta^2]=:e^{-2t}g_c,
\end{equation}
where $r=e^{-t}$ is the radial variable, $d\theta^2$ is the canonical metric on $\mathbb S^3$, and $g_c$ the cylindrical metric on $X=\mathbb R\times\mathbb S^3$. Consider the spherical harmonic decomposition of $\mathbb S^{3}$. For this, let $\mu_\ell$ and $Y_\ell^m$ be the eigenvalues and eigenfunctions for $-\Delta_{\mathbb S^{3}}$, respectively. This is,
$$\mu_\ell=\ell(\ell+2)\quad \text{and}\quad -\Delta_{\mathbb S^{3}} Y_\ell^m=\mu_\ell Y_\ell^m,\quad \ell=0,1,\ldots.$$
Then any function $u$ on $\mathbb R\times\mathbb S^{3}$ can be written as
$$u(t,\theta)=\sum_{\ell,m} u_\ell(t)Y_\ell^m(\theta), \quad t\in\mathbb R, \theta\in\mathbb S^3.$$
In order to write the Paneitz operator $P$ with respect to the metric $g_c$, we observe that $P_{g_c}$ diagonalizes under this eigenfunction decomposition.
Let $P^{(\ell)}$ its projection over the eigenspace $\langle Y_\ell^m\rangle$. Recalling again the conformal property \eqref{cylinder-metric}, we have that
\begin{equation}\label{Paneitz-cylinder}
\begin{split}
P^{(\ell)}u_\ell&=e^{-4t}(-\Delta_{\mathbb R^4})^2|_{\langle Y_\ell^m\rangle}u_\ell=r^4\Big(\partial_{rr}+\frac{3}{r}\partial_r -\frac{\mu_\ell}{r^2}\Big)\Big(\partial_{rr}+\frac{3}{r}\partial_r -\frac{\mu_\ell}{r^2}\Big)u_\ell\\
&=
\left(\partial_{tt}-\left(2+\ell\right)^2\right)\left( \partial_{tt}-\ell^2\right)u_\ell,
\end{split}
\end{equation}
after the change of variable $r=e^{-t}$.

We deal first with eigenfunctions for $P$ in $\mathbb R\times \mathbb S^3$ with its canonical metric $g_c$, that are $\varrho$-periodic in the variable $t$.  More precisely:

\begin{lemma}\label{lemma:periodic}
Non-constant $\varrho$-periodic eigenfunctions of $P^{(\ell)}$ are of the form
\begin{equation*}
u_\ell(t)
=c_1\cos(\tfrac{2\pi}{\varrho}t  )+c_2\sin(\tfrac{2\pi}{\varrho}t ),
\end{equation*}
for an eigenvalue
\begin{equation*}
    \lambda_\ell^c=\big[(2+2\ell+\ell^2)+(\tfrac{2\pi}{\varrho})^2\big]^2 -(4+8\ell+4\ell^2).
\end{equation*}
\end{lemma}
The proof will be postponed to the Appendix.\\

Now let us consider $P_{g_c}$ when $M$ is a ball of radius $r_1$. The corresponding boundary operators on the boundary $\Sigma=\mathbb S^3_{r_1}$, which in $t$ coordinates is $\Sigma=\{t=-\log r_1\}$, will be denoted by $B^1_{r_1}$ and $B^3_{r_1}$. From the conformal change $g_c=e^{2t}|dx|^2$ we have
\begin{equation*}
B^1_{r_1} u=e^{-t} B^1_0 u|_{r=r_1}=-\partial_t u|_{t=-\log r_1},
\end{equation*}
and
\begin{equation}\label{Panetiz-boundary:cylinder}
B^{3,(\ell)}_{r_1}u_\ell=e^{-3t} B^{3}_{0}u_\ell|_{\langle Y_\ell^m\rangle,r=r_1}=  \left\{ \partial_{ttt}-\left(3\mu_\ell+3\right)\partial_t\right\}u_\ell|_{t=-\log r_1},
\end{equation}
where we have denoted by $B^{3,(\ell)}_{r_1}$, $\ell=0,1,\ldots$ the projection over spherical harmonics of $B^3_{r_1}$.\\

We next take a new metric in $M$ given by $g_f=e^{2f}g_c$ for a radially symmetric conformal factor $f$. The Paneitz operator with respect to the metric $g_f$ on $M$ will be denoted by $P_f$, and the corresponding boundary operators on $\Sigma$ by $B^1_f$, $B^3_f$. By the conformal property of the operators, we have
\begin{align}
&P_f u=e^{-4f} P^{g_c},\nonumber\\
&B^{1}_f u= e^{-f} B^1_{r_1} u\big|_{t=-\log r_1} \nonumber\\
\label{Panetiz-boundary:cylinder-general-conformal-factor}
&B^{3,(\ell)}_{f}u_\ell=\left.e^{-3f} \{ \partial_{ttt}-\left(3\mu_\ell+3\right)\partial_t\right\}u_\ell\big|_{t=-\log r_1},
\end{align}

\subsection{Eigenvalues  for the (unit) ball }\label{subsection:eigenvalues-ball}

Let $M$ be the unit ball in $\mathbb R^4$, parameterized in cylindrical coordinates (here $t\in[0,+\infty)$). We take a conformally flat metric $g_f=e^{2f}g_c$, where $f$ only depends on the radial variable $t$. We normalize $e^ {f(0)}=1$, which is the same as normalizing the volume of the boundary sphere to $2\pi^2$. After projection over spherical harmonics we obtain,

\begin{lemma}\label{lemma:eigenvalues-ball}
In this setting, eigenvalues for \eqref{bi-laplacian-eq}-\eqref{eigenvalue-problem} are given by
$$\lambda_\ell=4(\ell+2)\quad \text{for }\quad \ell\geq 1,\quad \lambda_0=0,$$
with associated eigenfunctions
\begin{equation}\label{first-eigenfunction-ball}
u_\ell(t)Y_\ell^m\quad \text{for }\ell=1,2,\ldots,\quad\text{and}\quad u_0(t)=1.
\end{equation}
where
$$u_\ell(t)=\frac{(\ell+2)}{2\ell}e^{-\ell t}-\frac{e^{-(\ell+2)t}}{2}.$$
\end{lemma}

The proof will be postponed to the Appendix.

\subsection{Radially symmetric metrics in an annulus}

We now let $\calA_\rho=\{\rho \leq r\leq 1\}$ be an annulus in $\mathbb R^4$. In cylindrical coordinates we have $t\in[0,\tau]$, where $\tau=-\log \rho$. Take a conformally flat metric $g_f=e^{2f}g_c$, where $f=f(t)$ is a radially symmetric function. In addition, we impose the normalization
  \begin{equation}\label{normalization}
e^{3f(0)}+e^{3f(\tau)}=1.
\end{equation}
This again fixes the volume of the boundary to be $2\pi^2$. We set $\alpha=e^{3f(0)}$, for $\alpha\in(0,1)$.

As in the case of the ball, we decompose in spherical harmonics and look for eigenfunctions for \eqref{bi-laplacian-eq}-\eqref{eigenvalue-problem}
of the form $u_\ell(t) Y_\ell^m(\theta)$, $\ell=0,1,\ldots$:

\begin{lemma}\label{lemma:quadratic-eq}
Fix $\alpha\in(0,1)$. Let $\lambda$ be the an eigenvalue for the $\ell$-th projection, $\ell\geq 1$. Then $\lambda$ satisfies the quadratic equation
\begin{equation}a(\ell)\lambda^2+b(\ell)\lambda +c(\ell)=0 \label{eqn:quadratic for lambda},
\end{equation}
where
\begin{align*}
a(\ell)=&-2\ell(\ell+2)+2(\ell+1)^2\cosh(2\tau)-2\cosh((2\ell+2)\tau),\\
b(\ell)=&\frac{1}{\alpha(1-\alpha)}4\ell(\ell+1)(\ell+2) [(\ell+1)\sinh(2\tau)+\sinh((2\ell+2)\tau)],\\
c(\ell)=&-\frac{1}{\alpha(1-\alpha)}8\ell^2(\ell+1)^2(\ell+2)^2   [\cosh((2\ell+2)\tau)-\cosh(2\tau)].\end{align*}
For each $\ell\geq 1$, equation \eqref{eqn:quadratic for lambda} has exactly two  solutions $\lambda_\ell^-<\lambda_\ell^+$. For each $\lambda$ that solves \eqref{eqn:quadratic for lambda}, the corresponding eigenfunctions are the form $u_\ell^{\pm}(t)Y_\ell^m(\theta)$
for
$$u_\ell^{\pm}(t):=u_\ell^1+\frac{\lambda_\ell^{\pm}  \alpha}{4\ell(\ell+2)(\ell+1)} u_\ell^2,$$
where
\begin{align*}
u_\ell^1(t)=&(\ell+2)\sinh((\ell+2)\tau)\cosh(\ell t)-\ell\sinh(\ell \tau)\cosh((\ell+2)t),\\
u_\ell^2(t)=&[\ell\sinh(\ell\tau)-(\ell+2)\sinh((\ell+2)\tau)][(\ell+2)\sinh(\ell t)-\ell\sinh((\ell+2) t)]\\
&-\ell(\ell+2)
[\cosh(\ell\tau)-\cosh((\ell+2)\tau)][\cosh(\ell t)-\cosh((\ell+2)t)].
\end{align*}
For $\ell=0$ there are two eigenvalues: $\lambda_0^-=0$ (with just constant eigenfunctions) and
$$\lambda_0^+=\frac{4}{\alpha(1-\alpha)}\frac{\sinh(2\tau)}{1-\cosh(2\tau)+\tau\sinh(2\tau)}>0,$$
with eigenfunction
$$u_0^+(t)=\sinh(2 \tau)(\sinh (2t)-2t)+(1-\cosh(2\tau))\cosh(2t)+2(1-\alpha)(1-\cosh(2\tau)+\tau\sinh(2\tau))-1+\cosh(2\tau).$$

\end{lemma}

The proof is just computational and it is also postponed to the Appendix.\\

If  write
$$\lambda^-_\ell=\frac{-b(\ell)}{2a(\ell)}-\sqrt{\frac{b(\ell)^2}{4a^2(\ell)}-\frac{c(\ell)}{a(\ell)}}$$
it is clear that:

\begin{corollary} \label{cor:annulus model}
All the non-trivial eigenvalues associated to the eigenvalue problem \eqref{bi-laplacian-eq}-\eqref{boundary-condition}-\eqref{eigenvalue-problem} in $(\mathcal A_\rho, g_f)$ are strictly positive. In addition, the eigenspace corresponding to the zero eigenvalue consists only of constant functions.
\end{corollary}

Unfortunately, it is not possible to characterize the spectral gap for the operator since the calculations for a general $\tau$ are too complicated (even if elementary). However, we have strong numerical evidence for the following:

\begin{conjecture}
For each $\tau>0$, the sequence $\{\lambda^-_\ell\}$ is increasing in $\ell$.
\end{conjecture}

Note that a similar computation  is performed in \cite{Fraser-Schoen:annulus} for Steklov eigenvalues in 2 dimensions and in their situation the conjecture holds. In addition, they prove that for each $\alpha$ there is a value  $\tau^*(\alpha)$ such that for $\tau\leq \tau^*(\alpha)$ the smallest non-trivial eigenvalue is
given by $ \lambda^-_1$, while for $\tau\geq \tau^*(\alpha)$  we have that the smallest non-zero eigenvalue is $ \lambda^+_0$. In the result of \cite{Fraser-Schoen:annulus} there is an $\alpha^*$ such that for $\tau^*(\alpha^*)$ that the smallest eigenvalue is
maximized and in that case $ \lambda^-_1=\lambda^+_1=\lambda^+_0.$ We prove a partial result in that direction. Set  $\beta=\alpha(1-\alpha)$.

\begin{proposition} \label{prop: first eigenvalue}
Given $\beta\in\left(0,\frac{1}{4}\right)$, there are values $\tau^-$, $\tau^*$ and  $\tau^+$ such that: for $\tau^-$ it holds $\lambda_0^+>\lambda_1^-$, for $\tau^-$ we have  $\lambda_1^->\lambda_0^+$, and for $\tau^*$   we encounter $\lambda_1^-=\lambda_0^+$ .
\end{proposition}

\begin{proof}
For  $\ell=1$ we have
\begin{align*}
a(1)=&-4(\cosh(2\tau)-1)^2,\\
b(1)=&\frac{48}{\alpha(1-\alpha)}  \sinh(2\tau) [1+\cosh (2\tau)], \\
c(1)=&-  \frac{288}{\alpha(1-\alpha)}  [\cosh(4\tau)-\cosh(2\tau)].
\end{align*}
Then, the associated eigenvalue is
\begin{equation*}
\begin{split}
\lambda_1^-= \frac{1}{\alpha(1-\alpha)}  \Big[&6\frac{\sinh(2\tau)(1+\cosh (2\tau))}{(\cosh (2\tau)-1)^2}\\
&-6\sqrt{\frac{\sinh^2(2\tau)(1+\cosh(2\tau))^2}{(\cosh(2\tau)-1)^4}- 2\alpha(1-\alpha)\frac{  [\cosh(4\tau)-\cosh(2\tau)]}{(\cosh(2\tau)-1)^2}}\,\Big],
\end{split}
\end{equation*}
and it has multiplicity four (this is the number of spherical harmonics in 3 dimensions associated to $\ell=1$).\\

Now we take the quotient $\lambda_1^-/\lambda_0^+$, denoted by
\begin{equation*}
\begin{split}
F(\beta, \tau):=\frac{\lambda_1^-}{\lambda_0^+}=   \frac{3}{2}\Big[&\frac{\sinh(2\tau)(1+\cosh (2\tau))}{(\cosh (2\tau)-1)^2}\\
&-\sqrt{\frac{\sinh^2(2\tau)(1+\cosh(2\tau))^2}{(\cosh(2\tau)-1)^4}- 2\beta \frac{  [\cosh(4\tau)-\cosh(2\tau)]}{(\cosh(2\tau)-1)^2}}\,\Big]\\
&\quad\cdot\left(\frac{\sinh(2\tau)}{1-\cosh(2\tau)+\tau\sinh(2\tau)}\right)^{-1}.
\end{split}
\end{equation*}
A tedious, but straightforward computation reveals that
\begin{itemize}
\item $F(\beta, \tau)\to 0 $ as $\tau\to 0$.
\item  $F(\beta, \tau)\to \infty $ as $\tau \to \infty$

\end{itemize}

By  the continuity of $F(\beta, \tau)$  we conclude that for each $\beta$ there are values of  $\tau^-$, $\tau^*$ and  $\tau^+$  for which  $ F(\beta, \tau^-)<1$,  $ F(\beta, \tau^+)>1$ and  $ F(\beta, \tau^*)=1$, which concludes the proof.
\end{proof}

\begin{remark}
Note that for each value of $\beta$ we have two values of $\alpha$. This is because the problem is symmetric.
\end{remark}

\begin{remark}
Numeric computations strongly suggest that the function $F(\beta, \tau)$ in the proof above is increasing and the value $\tau^*$ is unique.
\end{remark}

\begin{remark}
If the conjecture above holds, then the smallest eigenvalue is either $\lambda_1^-$ or $\lambda_0^+$.
\end{remark}

\section{(Boundary) eigenvalue problems} \label{sec: boundary value}

In this section we study the eigenvalue problem \eqref{bi-laplacian-eq}-\eqref{boundary-condition}-\eqref{eigenvalue-problem}, which we recall here:
\begin{equation}\label{P1}
\left\{\begin{split}
&P_g u=0    \hbox{ in } M,\\
&B^1_g u=0\hbox{ on } \Sigma,\\
&B^3_g u=\lambda u\hbox{ on } \Sigma.
\end{split}\right.
\end{equation}
for $M$ a manifold with boundary satisfying the hypothesis in Theorem \ref{thm:positivity}.
For  \eqref{P1}, one can show that there exists a non-decreasing sequence of eigenvalues $\{\lambda_0^g,\lambda_1^g,\lambda_2^g,\ldots\}$. Corollary \ref{cor:positivity} characterizes the zero-eigenspace and yields positivity of the operator if $\Sigma$ has either one or two connected components.
In addition, recall that $\lambda_1^g$ is characterized by the Rayleigh quotient \eqref{Rayleigh-quotient1}. Since the energy $\mathcal E^M_g[u]$ is conformally invariant, the  quotient
 \eqref{Rayleigh-quotient1}  remains strictly positive when conformal transformations of $M$ are applied.

In what follows we prove Corollary \ref{cor:positivity} and Theorem \ref{thm:bounds-boundary}, by considering separately the cases that, either $M$ is conformally equivalent to a half-sphere $\mathbb S^4_+$, or an annulus $\mathcal A_\rho$ by Theorem \ref{thm:positivity}. This allows to reduce the study of problem \eqref{P1} to the model cases from the previous section.

\subsection{One boundary component. Proof of Corollary \ref{cor:positivity}  and Theorem \ref{thm:bounds-boundary}}

Let $M$ be conformally equivalent to a half-sphere $\mathbb S^4_+$. By stereographic projection, we can assume that   $M=\mathbb B^4$, $\Sigma=\partial \mathbb B^4=\mathbb S^3$, with a conformal metric $g=e^{2w}|dx|^2$.

We consider first the trivial setting, namely $w\equiv 0$. Lemma \ref{lemma:eigenvalues-ball} implies that $\lambda_0^g=0$ with $\Ker (B^3)=\{csts\}$ and $\lambda_1^g>0$. Nevertheless, this result can be proved directly by a simple integration by parts argument. Indeed, take the model $M=\mathbb R^4_+$, $\Sigma=\mathbb R^3$, with coordinates $(x_1,x_2,x_3,y)$, $y>0$, $x_1,x_2,x_3\in\mathbb R$. In this particular case we have
\begin{equation*}P=\Delta^2,\quad B^1=-\partial_y,\quad B^3=-\partial_y \Delta.
\end{equation*}
Let $\psi$ be any smooth solution to \eqref{P1}.
Integrating by parts we explicitly see that
\begin{equation*}
\lambda \int_\Sigma \psi^2 \, dx= \int_\Sigma \psi B^3 \psi\,dx=\int_M (\Delta \psi)^2\,dxdy\geq 0,
\end{equation*}
and it is zero if and only if $\Delta \psi=0$. Since we also have that  $\partial_y u=0
$ at $\Sigma$, we conclude that $\psi$ is constant up to the boundary, as desired.

We also remark that existence such solution $\psi$ can be proved by a standard minimization argument, while the uniqueness follows from the previous proof (since the constant obtained above would be 0).\\

For the case of a general manifold $M$ in Corollary \ref{cor:positivity} we recall again that the energy $\mathcal E_g^M[u]$ in  \eqref{Rayleigh-quotient1} is conformally invariant. In particular,
if $\tilde u$ is a minimizer of  \eqref{Rayleigh-quotient1} with respect to the metric  $g=e^{2w}|dx|^2$, we would have that
$$\lambda_1^g=\frac{\mathcal E_g^M[\tilde u]}{\int_\Sigma \tilde u^2 \,d\sigma_h}\geq \lambda_1^0\, \frac{\int_\Sigma \tilde u^2\, dv_0}{\int_\Sigma \tilde u^2 \,\sigma_h}>0 ,$$
where $\lambda^g_1$ is the first non-zero eigenvalue with respect to the metric $g$, $\lambda_1^0 $ is the first non-zero eigenvalue with respect to the flat metric and $dv_0, dv_g$ are volume elements with respect to the flat metric in the ball and the metric given by $g$, respectively.

To prove Theorem \ref{thm:bounds-boundary} we proceed as in the proof of Theorem \ref{thm:closed}. We use the unit ball model and assume that there exists a conformal embedding $\Phi:M\to \mathbb B^4$ satisfying  $\Phi(\Sigma)=\partial \mathbb B^4=\mathbb S^3$. In the two-dimensional case, the bound for the Steklov eigenvalue is obtained by using the coordinate functions of the embedding as test functions in the Rayleigh quotient \eqref{Rayleigh-quotient1}. In our setting we will use instead the (four) eigenfunctions for the first non-zero eigenvalue in the ball model and calculated in \eqref{first-eigenfunction-ball} for $\ell=1$. The precise expression is
$$U_m(r,\theta)=u_1(-\log r)Y_1^m(\theta),\quad m=1,2,3,4.$$
Next, on $M$ we set $\overline U_m=U_m \circ \Phi$.

Denote by $\Phi_m$, $m=1,2,3,4$ be the coordinate functions of the embedding. By a calibration argument as in \eqref{calibration}
we can assume that
\begin{equation}\label{calibration1}
\int_\Sigma \Phi_m\,d\sigma_h=0.
\end{equation}
Indeed, after choosing  a M\"obius transform of the boundary $\mathbb S^3$ in order to calibrate the center of mass \eqref{calibration1}, one takes its unique extension to a conformal transformation of $\mathbb B^4$ (which is known as the Poincar\'e extension, see \cite[Section 4.4]{Ratcliffe}).

In addition, $B_g^1 \overline U_m=0$ thanks to the covariance property \eqref{covariance} and the construction of $u_1$. Finally, noting that we have $u_1(0)=1$, this implies that   $\overline U_m=\Phi_m$ along $\Sigma$. We conclude  that $\overline U_m$ is an admissible test function in the Rayleigh quotient \eqref{Rayleigh-quotient1}.

We thus calculate
\begin{equation}\label{test-function}
 \lambda^g_1 \int_\Sigma \overline U_m^2\,d\sigma_h\leq  \mathcal E^M_g[\overline U_m]=  \mathcal E^{\Phi(M)}_{ g_{\mathbb{B}^4}}[U_m].
 \end{equation}
Adding on $m=1,2,3,4$ and recalling that $\sum_{m=1}^4 (Y_1^m)^2(\theta) =1$  we have
$$ \lambda^g_1 \vol(\Sigma)\leq \sum_m \mathcal E^{\Phi(M)}_{ g_{\mathbb{B}^4}}[U_m].$$
Next, since $U_m$ is an eigenfunction in the ball model,
$$\mathcal E^{\mathbb B^4}_{ g_{\mathbb{B}^4}}[U_m]=\lambda_1^0 \int_{\mathbb S^3} U_m^2\,d\sigma_{g_{\mathbb S^3}}.$$
Adding on $m$, taking into account that $\lambda_1^0=12$, we conclude
$$ \lambda^g_1 \vol(\Sigma)\leq  12\vol(\mathbb S^3)=24\pi^2.$$

\subsection{Two boundary components. Proof of Corollary \ref{cor:positivity}  and Theorem \ref{thm:bounds-boundary}}

 In the light of Theorem \ref{thm:positivity}, it is enough to take $M$ to be conformally equivalent to the annulus $\mathcal A_\rho=\{\rho\leq |x|\leq 1\}$ in $\mathbb R^4$.  Denote by $\Sigma_1$, $\Sigma_\rho$ the boundaries of $\Sigma$ corresponding to  $|x|=1$, $|x|=\rho$, respectively. In cylindrical variables, this annulus is conformally equivalent to  $[0,\tau]\times\mathbb S^3$ with the metric $g_c$. Such $\tau$ plays the role of the conformal modulus, in analogy to the two-dimensional case.

 First, Corollary \ref{cor:positivity} follows from Lemma \ref{lemma:quadratic-eq} and Corollary \ref{cor:annulus model}.

For the eigenvalue bound, we  follow the same idea as in the proof of Theorem 4.1 in \cite{Fraser-Schoen:annulus}. Let  $\tilde A_\rho$ be another annulus with the flat metric, having  boundaries $\tilde \Sigma_1$ and $\tilde \Sigma_\rho$ with the same boundary volume as $\Sigma_1$, $\Sigma_\rho$, respectively. Without loss of generality, we may rescale the metric by constant to achieve that the total boundary volume equals $\vol(\mathbb{S}^3)$ (this corresponds with the normalization \eqref{normalization}).

In the  $\tilde A_\rho$ model, we consider  the eigenfunction $u_0^+(t)$, given in Lemma \ref{lemma:quadratic-eq}, with eigenvalue $\lambda_0^+>0$, which only depends on $\tau$ and $\alpha$ (recall \eqref{normalization} and the definition of $\alpha$).

We pull back this function to the original manifold, setting $U_0(r)=u_0^+(-\log r)$ and $\overline U_0=U_0 \circ \Phi$. Note that $\overline U_0$ is a constant function at each of the two boundary components   $\Sigma_1$, $\Sigma_\rho$. Let us check that  $\overline U_0$ is a suitable test function in the Rayleigh quotient \eqref{Rayleigh-quotient1},  Calculate (with a slight abuse of notation)
\begin{equation*}
\begin{split}
\int_\Sigma \overline U_0\,d\sigma_h &=  \int_\Sigma  U_0\circ\Phi\,d\sigma_h=u_0^+(0)\vol(\Sigma_1)+u_0^+(\tau)\vol(\Sigma_\rho)\\
&=u_0^+(0)\vol(\tilde \Sigma_1)+u_0^+(\tau)\vol(\tilde \Sigma_\rho)=\int_{\tilde\Sigma} U_0\,dx=0.
\end{split}
\end{equation*}
Next, using $\bar U_0$ as a test function,
\begin{equation}\label{test-function-annulus}
 \lambda^g_1 \int_\Sigma \overline U_0^2\,d\sigma_h\leq  \mathcal E^M_g[\overline U_0]=  \mathcal E^{\Phi(M)}_{ g_{\text{flat}}}[U_0].
 \end{equation}
Now, on the one hand,
$$\int_\Sigma \overline U_0^2\,d\sigma_h=(u_0^+)^2(0)\vol(\Sigma_1)+(u_0^+)^2(\tau)\vol(\Sigma_\rho)$$
while, on the other hand,
$$E^{\Phi(M)}_{ g_{\text{flat}}}[U_0]=\lambda_0^+ \int_{\tilde\Sigma} U_0^2\,dx=\lambda_0^+\Big((u_0^+)^2(0)\vol(\Sigma_1)+(u_0^+)^2(\tau)\vol(\Sigma_\rho)\Big),$$
which yields statement \eqref{statement-annulus}.

Finally, Proposition  \ref{prop: first eigenvalue} implies that for $\tau\geq \tau^*$ (or equivalently, $\rho\leq \rho^*$) the bound is sharp.


\section{Appendix}

\begin{proof}[Proof of Lemma \ref{lemma:periodic}]
Use the formula \eqref{Paneitz-cylinder} for the Paneitz operador in cylindrical coordinates after spherical harmonic decomposition. We look for  $\varrho$-periodic solutions for the  constant coefficient ODE
\begin{equation}
\partial_{tttt}u_\ell-[(2+\ell)^2+\ell^2]\partial_{tt}u_\ell+(2+\ell)^2\ell^2u_\ell=\lambda_\ell u_\ell.
\end{equation}
For this, we  need to find purely imaginary roots of its characteristic polynomial
$$m^4-[(2+\ell)^2+\ell^2]m^2+(2+\ell)^2\ell^2-\lambda_\ell=0,$$
which must satisfy
\begin{equation*}
    m^2=2+2\ell+\ell^2-\sqrt{4+8\ell+4\ell^2+\lambda}=:-a^2\le 0.
\end{equation*}
Imposing periodicity ($a\varrho=2\pi$),
we arrive at the desired result.\\
\end{proof}

\begin{proof}[Proof of Lemma \ref{lemma:eigenvalues-ball}]
This is a straightforward calculation that we detail below.
 Taking into account that the Paneitz operator $P$ is conformally invariant, in the interior of the ball equation \eqref{bi-laplacian-eq} reduces to
$$\left(\partial_{tt}-\left(2+\ell\right)^2\right)\left( \partial_{tt}-\ell^2\right)u_\ell=0$$
In addition we require
\begin{equation}\label{boundary-condition0}
u'_\ell(0)=0
\end{equation}
(which corresponds to the condition \eqref{boundary-condition} $B^1(u_\ell)=0$), and $u$ is finite at infinity. Let $v_\ell=(\partial_{tt}-\ell^2)u_\ell$, then
$$\left(\partial_{tt}-\left(2+\ell\right)^2\right)v_\ell=0$$
so
$$v_\ell(t)=A e^{-(\ell+2)t}.$$
Hence, for $\ell\geq 1$,
$$u_\ell(t)=Ce^{-\ell t}+\frac{A}{4\ell+4}e^{-(\ell+2)t}.$$
Imposing the boundary condition \eqref{boundary-condition0} we have
$$u_\ell(t)=-\frac{A(\ell+2)}{4\ell(\ell+1)}e^{-\ell t}+\frac{A}{4\ell+4}e^{-(\ell+2)t}.$$
The eigenvalues of $B^3$  are given by $\left.\frac{B^3u_\ell}{u_\ell}\right|_{t=0}$, this is, recalling \eqref{Panetiz-boundary:cylinder-general-conformal-factor},
$$\lambda_\ell=(\ell+2)\frac{\ell^2-(\ell+2)^2}{-(\ell+2)+1}=4(\ell+2),\quad \text{for }\quad \ell\geq 1.$$

For $\ell=0$ the calculation is slightly different, since
$$u_0=C+\frac{A}{4}e^{-2t}.$$
Imposing the boundary condition implies that $A=0$, this is, $u_0$ is constant and $\lambda_0=0$, as expected.

\end{proof}

\begin{proof}[Proof of Lemma \ref{lemma:quadratic-eq}]
Fix $\ell\geq 1$. Using \eqref{Paneitz-cylinder} and the conformal property \eqref{conformal-Paneitz} we observe that the general solution $u_\ell$
can be written as
$$u_\ell(t)=A \sinh (\ell t)+B\cosh(\ell t)+C \sinh ((\ell+2)t)+D \cosh ((\ell+2)t).$$
Imposing that $B_1 u_\ell=0$ at the boundary of $\mathcal A_\rho$, it is easy to see that solutions are spanned by two functions that can be chosen as
\begin{align*}
u_\ell^1(t)=&(\ell+2)\sinh((\ell+2)\tau)\cosh(\ell t)-\ell\sinh(\ell \tau)\cosh((\ell+2)t)\\
u_\ell^2(t)=&[\ell\sinh(\ell\tau)-(\ell+2)\sinh((\ell+2)\tau)][(\ell+2)\sinh(\ell t)-\ell\sinh((\ell+2) t)]\\
&-\ell(\ell+2)
[\cosh(\ell\tau)-\cosh((\ell+2)\tau)][\cosh(\ell t)-\cosh((\ell+2)t)].
\end{align*}
Then, eigenfunctions can be written as $Au_\ell^1+B u_\ell^2$.\\

We can directly compute
\begin{align*}
\partial_{ttt} u_\ell^1(t)=&(\ell+2)\ell^3\sinh((\ell+2)\tau)\sinh(\ell t)-\ell(\ell+2)^3\sinh(\ell\tau)\sinh((\ell+2)t)\\
\partial_{ttt} u_\ell^2(t)=&\ell(\ell+2)[\ell\sinh(\ell\tau)-(\ell+2)\sinh((\ell+2)\tau)][\ell^2\cosh(\ell t)-(\ell+2)^2\cosh((\ell+2) t)]\\
&-\ell(\ell+2)
[\cosh(\ell\tau)-\cosh((\ell+2)\tau)][\ell^3\sinh(\ell t)-(\ell+2)^3\sinh((\ell+2)t)].
\end{align*}

The eigenfunction condition at $t=0$ can be written from \eqref{Panetiz-boundary:cylinder} as
$$e^{-3f(0)}(A \partial_{ttt} u_\ell^1(0)+B \partial_{ttt} u_\ell^2(0))=\lambda (Au_\ell^1(0)+B u_\ell^2(0)),$$
which implies
\begin{equation*}
\begin{split}
& e^{-3f(0)}B([\ell\sinh(\ell \tau)-(\ell+2)\sinh((\ell+2)\tau)][(\ell+2)\ell^3-\ell(\ell+2)^3]\\
&\qquad=\lambda A((\ell+2)\sinh((\ell+2)\tau)-\ell\sinh(\ell \tau)),
\end{split}
\end{equation*}
or equivalently,
\begin{equation}4B\ell(\ell+2)(\ell+1)=\lambda A e^{3f(0)}.\label{relation A and B}\end{equation}

For the eigenvalue condition at $\tau$ we need to take into account that the outward normal is reversed, so we have
$$-e^{-3f(\tau)}(A \partial_{ttt}u_\ell^1 (\tau)+B \partial_{ttt}u_\ell^2 (\tau))=\lambda(Au_\ell^1(\tau)+B u_\ell^2(\tau)).$$
Multiplying by $\lambda$ and using \eqref{relation A and B} we obtain
\begin{multline*}
- e^{-3f(\tau)}( 4\ell(\ell+2)(\ell+1) e^{-3f(0)} \partial_{ttt}u_\ell^1 (\tau)+ \lambda \partial_{ttt}u_\ell^2 (\tau))=\lambda 4\ell(\ell+2)(\ell+1) e^{-3f(0)} u_\ell^1(\tau)+\lambda^2 u_\ell^2(\tau)\end{multline*}
Or equivalently
\begin{equation} a(\ell)\lambda^2+ b(\ell)\lambda+c(\ell)=0 \label{eqn:quadratic for lambda app},\end{equation}
where
\begin{align*}
a(\ell)=& u_\ell^2(\tau),\\
b(\ell)=&4\ell(\ell+1)(\ell+2) e^{-3f(0)} u_\ell^1(\tau)+ e^{-3f(\tau)}\partial_{ttt}u_\ell^2 (\tau),\\
c(\ell)=&4\ell(\ell+1)(\ell+2)e^{-3(f(\tau)+f(0))}  \partial_{ttt}u_\ell^1 (\tau).
\end{align*}
From the previous computations we have, at $t=\tau$,
\begin{align*}
u_\ell^1(\tau)
=&(\ell+1)\sinh(2\tau)+\sinh((2\ell+2)\tau),\\
u_\ell^2(\tau)
=&-2\ell(\ell+2)+2(\ell+1)^2\cosh(2\tau)-2\cosh((2\ell+2)\tau),
\end{align*}
and for the derivatives
\begin{align*}
\partial_{ttt}u_\ell^1(\tau)
=&-2\ell(\ell+1)(\ell+2)[\cosh((2\ell+2)\tau)-\cosh(2\tau)],\\
\partial_{ttt}u_\ell^2(\tau)
=&4\ell(\ell+1)(\ell+2)[(\ell+1)\sinh(2\tau)+\sinh((2\ell+2)\tau)].
\end{align*}
To conclude,  recall our normalization \eqref{normalization} and set $\alpha=e^{3f(0)}$, for $\alpha\in(0,1)$, so $e^{-3f(\tau)}=(1-\alpha)^{-1}$.

Let us have a closer look at these eigenvalues now. By differentiating twice in $\tau$, we can easily check that $a(\ell)<0$ for $\tau>0$. In addition, $b(\ell)>0$ and $c(\ell)<0$. After some calculation, one can explicitly see that
the discriminant associated to \eqref{eqn:quadratic for lambda app} is strictly positive for $\alpha \in(0,1)$, hence there are two positive real roots, $0<\lambda_\ell^-<\lambda_\ell^+$.\\

Now we compute the eigenfunction for $\ell=0$. In that case we that the particular solutions are
\begin{align*}
&u^1_0(t)=\sinh(2 \tau)(\sinh (2t)-2t)+(1-\cosh(2\tau))\cosh(2t),
\\ &u^2_0(t)=1,
\end{align*}
so the general solution can be written as $Au^1_0+B u^2_0$. The eigenvalue condition at $t=0$ is equivalent to
\begin{equation}\label{eq10}8e^{-3 f(0)}A\sinh(2\tau)  =\lambda(A(1-\cosh(2\tau))+B).\end{equation}
On the other hand, at $t=\tau$
we have
\begin{equation*}
\begin{split}
 &-8e^{-3 f(\tau)}A\big\{\sinh(2 \tau)\cosh (2\tau)+(1-\cosh(2\tau))\sinh(2\tau)\big\}
 \\ &\quad=\lambda\big\{A\sinh(2 \tau)(\sinh (2\tau)-2\tau)+A(1-\cosh(2\tau))\cosh(2\tau)+B\big\}.
 \end{split}
 \end{equation*}
Subtracting both equations, we obtain a simple non-trivial eigenvalue:
%
$$\lambda_0^+=\frac{4}{\alpha(1-\alpha)}\frac{\sinh(2\tau)}{1-\cosh(2\tau)+\tau\sinh(2\tau)}.$$
Substituting in \eqref{eq10} we obtain
\begin{equation*}
B=\left\{2(1-\alpha)(1-\cosh(2\tau)+\tau\sinh(2\tau))-1+\cosh(2\tau)\right\}A,
\end{equation*}
and this finishes the proof of  Lemma \ref{lemma:quadratic-eq}.
\end{proof}

\noindent\textbf{Acknowledgements:} M.d.M. Gonz\'alez is supported by the Spanish government grant and MTM2017-85757-P and the Severo Ochoa program at ICMAT.
M. S\'aez is supported by the grant Fondecyt Regular 1190388.

\end{document}